\documentclass{amsart}
\usepackage{amsmath}
\usepackage{latexsym,amssymb,graphicx,verbatim}
\usepackage[mathscr]{eucal}
\usepackage{tikz}\usetikzlibrary{positioning,decorations.pathreplacing,patterns}

\newtheorem{thm}{Theorem}[section]
\newtheorem{theorem}[thm]{Theorem}
\newtheorem{lem}[thm]{Lemma}
\newtheorem{pro}[thm]{Proposition}

\newtheorem{cor}[thm]{Corollary}

\newtheorem{remark}[thm]{Remark}

\newcommand{\bigand}{\operatornamewithlimits{\hbox{\Large$\&$}}}

\newcommand{\refe}[1]{\stackrel{\eqref{#1}}{=}}

\newcommand{\dom}{\text{dom}}
\newcommand{\bea}{\begin{eqnarray*}}
\newcommand{\eea}{\end{eqnarray*}}

\newcommand{\ben}{\begin{enumerate}}
\newcommand{\een}{\end{enumerate}}

\newcommand{\bi}{\begin{itemize}}
\newcommand{\ei}{\end{itemize}}

\newcommand{\Rep}{\operatorname{Rep}}
\newcommand{\Choice}{\operatorname{Choice}}
\newcommand{\up}[1]{\textup{#1}}

\begin{document}

\title{Override and update}
\author{Marcel Jackson}
\address{Department of Mathematics and Statistics, La Trobe University, Victoria 3086,
Australia}
\email{m.g.jackson@latrobe.edu.au}
\author{Tim Stokes}
\address{Department of Mathematics and Statistics,
University of Waikato, Hamilton, New Zealand}
\email{tim.stokes@waikato.ac.nz}

\thanks{The first author was partially supported by ARC Future Fellowship FT120100666}
\begin{abstract}
Override and update are natural constructions for combining partial functions, which arise in various program specification contexts.  We use an unexpected connection with combinatorial geometry to provide a complete finite system of equational axioms for the first order theory of the override and update constructions on partial functions, resolving the main unsolved problem in the area.
\end{abstract}

\maketitle

\section{Introduction}

The constructions of override and update on partial functions were analysed by Berendsen, Jansen, Schmaltz, Vaandrager~\cite{minusover}, with a view to giving a complete system of axioms.  Formally, let $\mathcal{P}(X,Y)$ be the set of all functions having domains in the non-empty set $X$ and mapping into non-empty $Y$; these are to be viewed as modelling programs with inputs from $X$ and outputs in $Y$, with partiality reflecting the fact that the programs may abort or not halt for certain input values from $X$.  The {\em override} operation is defined\footnote{We here use the notation ``$\sqcup$" rather than $\triangleright$ as in \cite{minusover}, as $\triangleright$ conflicts with some well-established notation in the algebraic theory of functions, and $\sqcup$ has some current usage as override, under the name  ``preferential union" \cite{modrest}.} 
in \cite{minusover} as follows: for all $f,g\in \mathcal{P}(X,Y)$, 
\[
(f \sqcup g)(x):=\begin{cases} f(x)&\mbox{ if }x\in \dom(f),\\
g(x)&\mbox{ if } x\in \dom(g)\backslash\dom(f),\\
\mbox{ undefined}&\mbox{ otherwise}.
\end{cases}
\]

The operation of {\em update} $f[g]$ is then defined in \cite{minusover} to be the restriction of $g\sqcup f$ to the domain of $f$.    These constructions are natural enough in their own right, but are also widely encountered in formal methods for program specification and correctness, such as in the the specification language  Z \cite{spi} and VDM-SL \cite{jon} amongst others; see \cite{minusover} for further examples.  The operations also have a natural interpretation as constructions on ``patterned'' Venn diagrams, where each partial function $f:X\to Y$ is thought of as a region of definition (a set, corresponding to $\dom(f)$) and a ``pattern'' or ``colour'' (the action: thinking of $Y$ as a set of colours, each point in $\dom(f)$ is assigned a colour by $f$, which collectively pattern the region $\dom(f)$).  The  $\sqcup$ operation returns the union of the regions, with the pattern of the first region dominating the pattern of the second, while the $\underline{\phantom{a}}[\underline{\phantom{a}}]$ operation returns the first region, but with the pattern of the second region dominating where it overlaps; see Figure \ref{fig:1}.  This interpretation can also be recast in terms of choice functions on powersets -- see Sections \ref{sec:basics} and \ref{sec:update} -- and in fact represents the free algebra for the class of representable algebras; see Theorem \ref{thm:venn}.
A further manifestation of $\sqcup$ arises in combinatorial geometry, but we postpone discussion of this until Section~\ref{sec:override}.

\begin{figure}
\begin{tikzpicture}
\draw [white] (-2,0) rectangle (-1,0);
\draw [ultra thick, pattern=north west lines] (0,0) circle [radius=1];
\draw [ultra thick, pattern=north east lines] (1.25,0) circle [radius=1];
\node at (0,-1.5) {$A$};
\node at (1.25,-1.5) {$B$};
\end{tikzpicture}\qquad 
\begin{tikzpicture}
\draw [ultra thick, pattern=north east lines] (1.25,0) circle [radius=1];
\draw [fill = white] (0,0) circle [radius=1];
\draw [ultra thick, pattern=north west lines] (0,0) circle [radius=1];
\node at (.625,-1.5) {$A\sqcup B$};
\end{tikzpicture}\qquad
\begin{tikzpicture}
\scope
\clip 
      (1.25,0)  circle  (1);
\fill [pattern=north east lines] (0,0) circle (1);
\endscope
\scope
\clip (-1,-1) rectangle (1,1)
      (1.25,0)  circle  (1);
\fill [pattern=north west lines] (0,0) circle (1);
\endscope
\draw [ultra thick] (0,0) circle (1);
\scope
\node at (0,-1.5) {$A[B]$};
\clip 
      (0,0)  circle  (1);
\draw [ultra thick]  (1.25,0)  circle  (1);
\endscope
\end{tikzpicture}
\caption{Patterned Venn diagrams and their combination via $\sqcup$ and $\rule{5pt}{.2pt}[\rule{5pt}{.2pt}]$.
}\label{fig:1}
\end{figure}
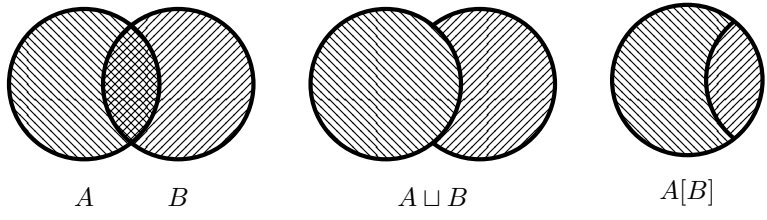

A complete system of axioms in the signature $\{\sqcup,\underline{\phantom{a}}[\underline{\phantom{a}}]\}$ is left as an open problem in~\cite{minusover}.  A possible system is presented:
\begin{align}
x&=x[x\sqcup y]\label{eq:absorb}\\
x\sqcup y &= y[x]\sqcup x\label{eq:absorb2}\\
x[y][z] &= x[ z\sqcup y]\label{eq:leftdist}\\
(x\sqcup y)[z] &= x[z]\sqcup y[z]\label{eq:rightdist}
\end{align}
along with associativity and idempotence of $\sqcup$.  We have in \cite[pp.~149]{minusover}:
\begin{quotation}
``Are these laws complete? If not, which laws should be added? Does there exist a finite equational axiomatization of override and update without auxiliary operators?''
\end{quotation}
The issue is subsequently revisited by Cvetko-Vah, Leech and Spinks \cite{CLS} where similar questions are reiterated; see \cite[pp.~262]{CLS} (the penultimate paragraph of Section 6).

In this article we will resolve these questions by showing 
\begin{itemize}
\item  there is a valid equational law that does not follow from the  speculated axiomatisations in both \cite{minusover} and \cite{CLS} and 
\item  with the addition of this extra law, a full completeness result can be obtained.
\end{itemize}  
In particular, the full first order theory of override and update on functions can be completely axiomatised by a finite set of equations.

After some preliminary development and investigation, we uncover the unexpected connection with combinatorial geometry and semigroup theory and use this to note an infinite axiomatisation for the signature of override alone (Corollary~\ref{cor:sqcupaxioms}).  No finite axiomatisation is possible for this signature.  We then show how to extend this to an infinite complete axiomatisation for the main signature of interest $\{\sqcup,\underline{\phantom{a}}[\underline{\phantom{a}}]\}$, but then provide an argument to show that these infinitely many laws are consequences of a finite set (Theorem \ref{thm:overrideupdate}).  During this proof we make frequent reference to the automated theorem prover {\em Prover9} and its companion counterexample program {\em Mace4} (see \cite{P9M4}), although only Remark \ref{rem:redundancies}, which observes redundancies in axiom systems, states facts that are not proved directly in the article.  There is also one counterexample obtained from {\em Mace4}, though this is in principle human-checkable.

Finally, in Section \ref{sec:update} we provide a description of the expressive power of update, from which various other similar results follow.  In terms of patterned Venn diagrams, the result shows that any conceivable patterning of the subregions of a single global region can  be achieved by applications of update alone.

\section{Preliminaries: other operations and algebras of functions}\label{prelim:otheroperations}

Throughout, we refer to elements of $\mathcal{P}(X,Y)$ as ``functions'' (rather than partial functions or maps), and to subalgebras of $\mathcal{P}(X,Y)$ (with respect to some signature of operations $\tau$) as $\tau$-algebras of functions.

The authors of \cite{minusover} embed the signature $\{\sqcup,\underline{\phantom{a}}[\underline{\phantom{a}}]\}$ into a more expressive signature $\{\sqcup,-\}$, where the ``minus'' operation is defined by 
\[
f-g=g\text{ where $f$ is undefined}.
\]
The operation of \emph{restrictive multiplication}:
\[
f\mathbin{@} g= f\text{ where $g$ is defined}
\]
 arises as a term function by way of $f\mathbin{@} g=f-(f-g)$, and then $\underline{\phantom{a}}[\underline{\phantom{a}}]$ arises by way of $f[g]=(g\sqcup f)\mathbin{@} f$.  In \cite{minusover} restrictive multiplication is refered to as intersection, but we prefer to reserve this name for actual intersection, which applies to functions as
\[
f\cap g= \{(x,y)\in X\times Y\mid (x,y)\in f\text{ and }(x,y)\in g\}
\]
and which goes back at least to the work of Garvac$'$ki\u{\i} \cite{gar}.
Restrictive multiplication itself also has interest in its own right, and has been considered as far back as the pioneering work of Vagner~\cite{vagner}.  The symbol $\triangleright$ is quite standard in abstract treatments of restrictive multiplication (see Schein \cite{sch70} for one of many examples), albeit in reverse form as $f\mathbin{@}g=g\triangleright f$.  As noted in the introduction, there is potential confusion here, as the $\triangleright$ symbol is used in \cite{minusover} to denote override, though in the present article we have adopted $\sqcup$.  We note that restrictive multiplication and minus also have natural interpretations in terms of patterned Venn diagrams; see Figure \ref{fig:2}.  (Amusingly, the Ti\textit{k}Z command \verb!\clip!  used to produce the diagram for restrictive multiplication, and other diagrams in this article, is itself an example of the operation of restrictive multiplication: the command restricts a patterned circle centred to the left to the domain of a circle to the right.)
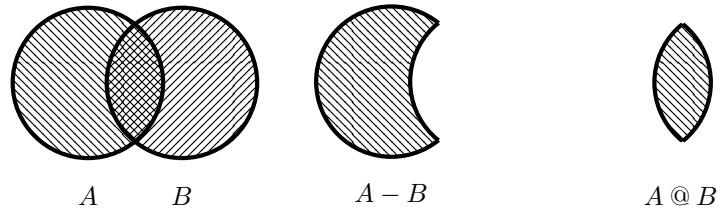
\begin{figure}[h]
\begin{tikzpicture}
\draw [white] (-2,0) rectangle (-1,0);
\draw [ultra thick, pattern=north west lines] (0,0) circle [radius=1];
\draw [ultra thick, pattern=north east lines] (1.25,0) circle [radius=1];
\node at (0,-1.5) {$A$};
\node at (1.25,-1.5) {$B$};
\end{tikzpicture}\qquad \begin{tikzpicture}
\scope
\clip (-1.05,-1.05) rectangle (1.05,1.05)
      (1.25,0)  circle  (1);
\fill [ultra thick, pattern=north west lines] (0,0) circle (1);
\draw  [ultra thick] (0,0) circle (1);
\endscope
\scope
\clip (0,0) circle (1);
\draw  [ultra thick] (1.25,0)  circle  (1);
\endscope
\node at (0,-1.5) {$A-B$};
\end{tikzpicture}\begin{tikzpicture}
\scope
\clip (1.25,0)  circle  (1);
\draw  [ultra thick, pattern=north west lines] (0,0)  circle  (1);
\endscope
\scope
\clip (0,0)  circle  (1);
\draw  [ultra thick] (1.25,0)  circle  (1);
\endscope
\node at (0.6,-1.5) {$A\mathbin{@} B$};
\end{tikzpicture}
\caption{The combination of patterned Venn diagrams via $-$ and $\mathbin{@}$}\label{fig:2}
\end{figure}

In \cite[\S6]{minusover} it is shown that the pair of operations $\{-,\sqcup\}$ is expressive enough to produce all possible patternings over any Boolean combination of domains; here Boolean complementation is interpreted within the union of the master regions, and each subregion has a pattern consistent with that of a master region containing it.   We revisit this in Section \ref{sec:update} of the present article, by showing that update alone is already capable of all possible patternings within a given domain.

A succinct set of equations is shown in \cite{minusover} to be complete for the equational theory of functions with $\{\sqcup,-\}$, therefore obtaining laws to reason with $\{\sqcup, \underline{\phantom{a}}[\underline{\phantom{a}}]\}$.  We refer to the notion of completeness for equations as \emph{weak completeness}, reserving \emph{completeness} for the situation where all first order properties are consequences (rather than only equational properties).  
 In general the setting of weak completeness still allows for a high level of incompleteness on nonequational properties: witness  for example, Kozen's axiom system KA which is complete for the equational theory of regular languages \cite{koz}, versus the strong incompleteness result given in \cite{koz2}.  In the particular case of  $\{\sqcup,-\}$, however, it was subsequently shown by Cvetko-Vah, Leech and Spinks  \cite{CLS}  that full completeness holds anyway, as the system turns out to be term-equivalent to certain types of skew Boolean algebras previously considered by Leech in \cite{LeechSBA}, where a complete axiomatisation of the algebras was given, and shown to be a variety.  In fact both the signatures $\{\sqcup, \underline{\phantom{a}}[\underline{\phantom{a}}]\}$ and $\{\sqcup,-\}$ are already expressible in the richer theory described in Cirulis \cite{cirulis} or indeed in the authors' work \cite{modrest}, where completeness results are also obtained.

In subsequent work we will consider many of the remaining combinations of $\{\sqcup,\underline{\phantom{a}}[\underline{\phantom{a}}], \mathbin{@},-\}$, as well as operations of functional intersection $\cap$, set difference and even composition.  In these cases we are able to find a unified approach that can be adjusted to obtain complete axiomatisations in each case, provided that restrictive multiplication is expressible.  The original signature of interest $\{\sqcup,\underline{\phantom{a}}[\underline{\phantom{a}}]\}$, which is the subject of the present article, distinguishes itself from these by seemingly requiring a different approach.

\section{Preliminaries: representation basics and finite generation}\label{sec:basics}

In this preliminary section we observe some elementary facts concerning \emph{representable algebras}: algebras that are faithfully representable as algebras of functions.  Throughout, when $\tau$ is a signature of operations on functions, by a \emph{$\tau$-algebra of functions} we mean an algebra whose elements are functions, and whose fundamental operations are the operations in $\tau$.  We let $\Rep(\tau)$ denote the class of algebras isomorphic to $\tau$-algebras of functions.

The following easy observation (whose proof is omitted) details some of the flexibility available in choosing the set $Y$ when representing in an algebra of functions $\mathcal{P}(X,Y)$.

\begin{lem}\label{lem:disjoint}
Let $X, Y$ be nonempty sets and $Z$ denote the disjoint union $X\times\{1\}\cup Y\times \{2\}$.    For any subset $\tau\subseteq\{\sqcup,\underline{\phantom{a}}[\underline{\phantom{a}}],\mathbin{@},-\}$, the algebra $\langle \mathcal{P}(X,Y),\tau\rangle$ is isomorphic to $\langle \mathcal{P}(X\times\{1\},Y\times \{2\}),\tau\rangle$ which is a subalgebra of $\langle \mathcal{P}(Z,Z),\tau\rangle$.  
\end{lem}

The following lemma is also trivial.

\begin{lem}\label{lem:functionquotient}
Let $\tau\subseteq\{\sqcup,\underline{\phantom{a}}[\underline{\phantom{a}}],\mathbin{@},-\}$.  Let $X,Y$ be sets, let $\theta$ be an equivalence relation on $Y$ and let $X'$ be a subset of $X$.  
\begin{enumerate}
\item Any function $f$ in $\mathcal{P}(X,Y)$ determines a function $f/\theta$ in $\mathcal{P}(X,Y/\theta)$ via $(f/\theta):x\mapsto f(x)/\theta$ and \up(as a $\tau$-algebra of functions\up) $\mathcal{P}(X,Y/\theta)$ is a quotient of $\mathcal{P}(X,Y)$.  
\item Any function $f$ in $\mathcal{P}(X,Y)$ determines a function $f'$ in $\mathcal{P}(X',Y)$ by restricting its domain to $X'$ and \up(as a $\tau$-algebra of functions\up) $\mathcal{P}(X',Y)$ is a quotient of $\mathcal{P}(X,Y)$.  
\end{enumerate}
\end{lem}

The following algebra ${\bf 3}$ and its reducts will play a prominent role.
\[
\begin{tabular}{c|ccc}
$\cdot$&0&1&$2$\\
\hline
0&0&1&$2$\\
1&1&1&1\\
$2$&$2$&$2$&$2$
\end{tabular}\qquad
\begin{tabular}{c|ccc}
$-$&0&1&$2$\\
\hline
0&0&0&$0$\\
1&1&0&0\\
$2$&$2$&$0$&$0$
\end{tabular}
\]
The algebra is isomorphic to $\langle\mathcal{P}(\{x\},\{+,-\});\sqcup,-\rangle$, where $0$ is the empty function, $1$ maps $x$ to $+$ and $2$ maps $x$ to $-$.  It is shown in \cite{CLS} that the variety generated by~${\bf 3}$ coincides with the quasivariety generated by ${\bf 3}$, which in turn is precisely the class of representable $\{\sqcup,-\}$-algebras.  We now give an alternative proof of this second fact, extending it to all reduct signatures as well.

Let $\tau\subseteq\{\sqcup,\mathbin{@},-,\underline{\phantom{a}}[\underline{\phantom{a}}]\}$, and let ${\bf 3}_\tau$ be the reduct of ${\bf 3}$ to~$\tau$.  In the present article we are interested primarily in the signature $\{\sqcup\}$ and $\{\sqcup,\underline{\phantom{a}}[\underline{\phantom{a}}]\}$, but the other cases are used in subsequent work and it is easier to prove one general result than to prove two specific ones.  The arguments also work for term reducts of $\{\sqcup,\mathbin{@},-,\underline{\phantom{a}}[\underline{\phantom{a}}]\}$ as well as reducts, but we omit this more general statement for notational brevity.  There is also an extension to signatures including $\cap$, however we do not need this here and it requires replacement of ${\bf 3}$ by an infinite family of structures.

\begin{thm}\label{thm:3tau}
For any set of operations $\tau\subseteq\{\sqcup,\mathbin{@},-,\underline{\phantom{a}}[\underline{\phantom{a}}]\}$ and any sets $X,Y$, the algebra ${\bf S}=\langle\mathcal{P}(X,Y),\tau\rangle$ is isomorphic to a subalgebra of a direct power of ${\bf 3}_\tau$, and hence $\Rep(\tau)$ coincides exactly with the class $\mathsf{SP}({\bf 3}_{\tau})$ of isomorphic copies of subalgebras of direct powers of ${\bf 3}_\tau$. 
\end{thm}
\begin{proof}
First observe that as ${\bf 3}_\tau$ is representable, so too are all members of $\mathsf{SP}({\bf 3}_{\tau})$ (see \cite[pp.~1062]{modrest} for an example of the simple argument).  Thus it remains to show that all representable algebras  lie in $\mathsf{SP}({\bf 3}_{\tau})$.   
 For this, it suffices to show that for $f\neq g$ in $\mathcal{P}(X,Y)$ there is a $\tau$-homomorphism $\phi:{\bf S}\to{\bf 3}_\tau$ (preserving $\sqcup$ as~$\cdot$) such that $\phi(f)\neq\phi(g)$.  Indeed, in this case we have that $\langle\mathcal{P}(X,Y),\tau\rangle$ embeds into $({\bf 3}_\tau)^{\operatorname{hom}({\bf S},{\bf 3}_\tau)}$ by the map sending $h\in \mathcal{P}(X,Y)$ to the tuple $\varepsilon_h(\phi)=\phi(h)$ (for each $\phi\in \operatorname{hom}({\bf S},{\bf 3}_\tau)$).

So let $f\neq g$, and consider some point $x\in X$ where $f$ and $g$ disagree.  Without loss of generality, assume that $f$ is defined at $x$ and either $g$ is not defined at $x$, or the value of $g$ is distinct from that of $f$ at $x$.  Define an equivalence $\theta$ by
$f_1\mathrel{\theta}f_2$ if 
\begin{align*}
&f_1,f_2\text{ are defined at $x$ but take different values to $f(x)$,}\\
&f_1,f_2\text{ are undefined at $x$,}\\
&f_1(x)=f_2(x)=f(x)\text{ otherwise.}
\end{align*}
This equivalence is a congruence in the signature $\{\sqcup,\mathbin{@},-,\underline{\phantom{a}}[\underline{\phantom{a}}]\}$, and hence also in the signature $\tau$: indeed it is a combination of the two items in Lemma \ref{lem:functionquotient} (restrict to domain $\{x\}$, then identify all points in the range not equal to $f(x)$).   The quotient ${\bf S}/\theta$ is isomorphic to a subsemigroup of ${\bf 3}_\tau$ (which is an isomorphism unless one of the three cases in the definition of $\theta$ does not occur).  Also, $f$ is not $\theta$-related to $g$, so that the required separation property holds.
\end{proof}

A well-known theorem of Birkhoff (see \cite[Corollary II.11.10]{bursan} for example) shows that the free algebras for an $\mathsf{SP}$-closed class are contained within the class, and hence the free algebras for $\Rep(\tau)=\mathsf{SP}({\bf 3}_{\tau})$ are isomorphic to algebras of functions.  As we now show, the free agebras have a natural representation as choice functions on families of subsets of a set.  Recall that a \emph{choice function} for a family $\{S_i\mid i\in I\}$ of sets is a map $\gamma:\{S_i\mid i\in I\}\to \bigcup_{i\in I}S_i$ with the property $\gamma(S_i)\in S_i$ for each $i\in I$.  For any nonempty set $S$, the set of all choice functions with domains equal to nonempty subsets of the powerset $\wp(S)$ is easily seen to form a $\tau$-subalgebra of $\mathcal{P}(\wp(S),S)$, which we will denote by $\Choice(S)$.  The next theorem shows that the $S$-generated free algebra for $\operatorname{Rep}(\tau)$ can be represented in $\Choice(S)\leq \mathcal{P}(\wp(S),S)$, with the constant choice functions as free generators.  
\begin{thm}\label{thm:venn}
Fix a signature $\tau\subseteq\{\sqcup,\mathbin{@},-,\underline{\phantom{a}}[\underline{\phantom{a}}]\}$ and let  $F_\tau(S)$ denote the free algebra in $\mathsf{SP}({\bf 3}_\tau)$ with free generators $S$.    Then $F_\tau(S)$ embeds into $\Choice(S)$ via the map sending each $s\in S$ to the choice function with domain $\{A\subseteq S\mid s\in A\}$ and definition $A\mapsto s$.
\end{thm}
The proof is via a series of applications of Lemma \ref{lem:functionquotient}, but is postponed until Section \ref{sec:update} to maintain the flow of these preliminary sections.
Theorem \ref{thm:venn} shows that Venn diagram interpretations of $\tau$ terms are simply a diagrammatic visualisation of the free algebras (where each individual pattern denotes the generator chosen for that subregion).  However, Theorem~\ref{thm:venn} itself does not describe exactly which choice functions appear in $F_\tau(S)$, only what the free generators are.  When $\tau$ has the full strength of $\{\sqcup,\mathbin{@},-,\underline{\phantom{a}}[\underline{\phantom{a}}]\}$, it is relatively straightforward to see that $F_\tau(S)$ coincides with $\Choice(S)$, and a result equivalent to this is given in Berendsen et al.~\cite[Proposition~14]{minusover}.  In contrast, the operation of $\underline{\phantom{a}}[\underline{\phantom{a}}]$ preserves domains, so for this signature, all choice functions in $F_{\{\underline{\phantom{a}}[\underline{\phantom{a}}]\}}(S)$ (as a subalgebra of  $\Choice(S)$) must have domains identical to the domain of a free generator.  We return to this in Section \ref{sec:update} where we show that $F_\tau(S)$ consists of all such choice functions; similar classifications for remaining signatures in $\{\sqcup,\mathbin{@},-,\underline{\phantom{a}}[\underline{\phantom{a}}]\}$ then follow.

We did not describe relational semantics for the operations in $\{\sqcup,\mathbin{@},-,\underline{\phantom{a}}[\underline{\phantom{a}}]\}$, however the given functional definitions extend to the case of relations (replacing statements such as $f(x)=y$ by $(x,y)\in f$).
\begin{remark}\label{rem:binrel}
The proof of Theorem \ref{thm:3tau} continues to hold if binary relations are considered instead of functions.  Thus there is no algebraic distinction between the relational case and the functional case for reducts of $\{\sqcup,\mathbin{@},-,\underline{\phantom{a}}[\underline{\phantom{a}}]\}$.
\end{remark}

\subsection{Abstract definability}

We now expand on some of the basic term function relationships between the operations as in the introduction.  Each of the operations has been given a natural definition in the language of algebras of functions, but we now observe that sometimes this definition can be made solely in terms of the abstract algebraic properties.  Rather than introduce all of the necessary terminology to make this precise, observe that each of the operations has been defined by formul{\ae} in the multi-sorted language of concrete function.  But often we will discover that some of the operations may be defined in terms of the others, without recourse to the full language of function.  All of the following concepts can be generalised from function to other forms of binary (and higher arity) relations.  

Suppose that $\{\divideontimes_i\mid i\in I\}$ is a signature of operation and relation symbols, but with each $\divideontimes_i$ having some agreed interpretation as an operation or relation on functions (more formally we would associate each $\divideontimes_i$ with a formula in the multi-sorted language of concrete algebras of functions).  Similarly, let $\divideontimes$ (of arity $n$, say) be an operation, with some agreed definition on functions.  We say that $\divideontimes$ is \emph{abstractly definable} from $\{\divideontimes_i\mid i\in I\}$ (for functions) if there is a first order formula $\phi(x,x_1,\dots,x_n)$ in the signature $\{\divideontimes_i\mid i\in I\}$ such that in any concrete algebra of functions in the signature $\{\divideontimes_i\mid i\in I\}$, we have $h=\divideontimes(f_1,\dots,f_n)$ if and only if $\phi(h,f_1,\dots,f_n)$ holds.  Note that we do not require that the algebra be closed under $\divideontimes$.  The simplest instance of abstract definability is when $\divideontimes$ is simply a term function in $\{\divideontimes_i\mid i\in I\}$.  Thus for example, $f[g]=(g\sqcup f)\mathbin{@}f$ so that $\underline{\phantom{a}}[\underline{\phantom{a}}]$ is a term function in $\{\mathbin{@},\sqcup\}$ and the formula $\phi(x,x_1,x_2)$ defining $\underline{\phantom{a}}[\underline{\phantom{a}}]$ is simply $x= (x_2\sqcup x_1)\mathbin{@} x_1$.  In this situation, every functional model of $\{\divideontimes_i\mid i\in I\}$ is automatically a model of~$\divideontimes$.  A familiar example of the more general form of abstract definability (albeit concerning algebras of permutations rather than algebras of functions) is the definability of inverse in semigroups of permutations by way of composition and identity: the defining formula $\phi(x,x_1)$ is $x_1x=1\And xx_1=1$.  Earlier work of the authors and their coauthors \cite[\S3.3]{modrest}, \cite[\S3.4]{HJSz}, used examples of abstract definability to extend a representations for function composition and the antidomain operations to include the operation $\sqcup$. The following lemma gives a further example that plays an important role in our main proof.

\begin{lem}\label{lem:updatedef1}
The update operation $\underline{\phantom{a}}[\underline{\phantom{a}}]$ is abstractly definable from override $\sqcup$ by the  formula $\phi_1(x,x_1,x_2)$ consisting of the conjunction of the following equalities\up:
\begin{align}
x\sqcup x_{1}&= x\textup{ and } x_1\sqcup x = x_1\quad \textup{($x$ has same domain as $x_1$)}\label{eq:samedomain}\\
x\sqcup x_2 &=x_2\sqcup x  \quad\textup{ \textup($x$ agrees with $x_2$ where both $x$ and $x_2$ are defined\ldots \textup)}\label{eq:agree}\\
x_2\sqcup x_1&=x_2\sqcup x\quad\textup{ \textup(\ldots and with $x_1$ where $x$ is defined and $x_2$ is not\textup)}\label{eq:agree2}
\end{align}
\end{lem}
\begin{proof}
The parenthetical remarks are easily verified as true statements when $x,x_1,x_2$ are functions, and $x=x_1[x_2]$.  It is also trivial that these properties uniquely define the value of $x_1[x_2]$, so any function $x$ satisfying the equalities  must coincide with $x_1[x_2]$.
\end{proof}

The following theorem demonstrates the utility of abstract definability in finding axiomatisations.  We include the easy proof for completeness.
\begin{thm}\label{thm:abstractdef}
If $\divideontimes$ is abstractly definable by $\phi(x,x_1,\dots,x_n)$ from $\{\divideontimes_i\mid i\in I\}$ for functions and $\Sigma$ is a sound and complete axiomatisation for 
$\Rep(\{\divideontimes_i\mid i\in I\})$, then $\Sigma\cup\{\forall x_1\dots\forall x_n\,\phi(\divideontimes(x_1,\dots,x_n),x_1,\dots,x_n)\}$ is a sound and complete  axiomatisation for 
$\Rep(\{\divideontimes_i\mid i\in I\}\cup\{\divideontimes\})$.
\end{thm} 
\begin{proof}
Certainly $\Sigma\cup\{\forall x_1\dots\forall x_n\,\phi( \divideontimes(x_1,\dots,x_n),x_1,\dots,x_n)\}$ is a sound set of laws for $\{\divideontimes_i\mid i\in I\}\cup\{\divideontimes\}$ on functions, as we assumed that $\phi$ defined $\divideontimes$ from $\{\divideontimes_i\mid i\in I\}$.  For completeness, we need to show that any model is representable.  Now any model of $\Sigma\cup\{\forall x_1\dots\forall x_n\,\phi( \divideontimes(x_1,\dots,x_n),x_1,\dots,x_n)\}$ is a model of $\Sigma$, and as $\Sigma$ is complete for representability of the operations $\{\divideontimes_i\mid i\in I\}$, it follows that there is a representation in this signature.  But $\divideontimes$ is abstractly definable by $\phi(x,x_1,\dots,x_n)$, so that law $\forall x_1\dots\forall x_n\,\phi( \divideontimes(x_1,\dots,x_n),x_1,\dots,x_n)$ ensures that $\divideontimes$ is also correctly represented, as required.
\end{proof}

\section{Starting from override}\label{sec:override}
We now explore a connection between override and combinatorial geometry, leading to an inherently infinite axiomatization.

Let $\mathcal{H}$ denote a finite set of hyperplanes in $\mathbb{R}^n$ (that is, subspaces of dimension $n-1$), each including the origin.  Each $H\in \mathcal{H}$ partitions the space $\mathbb{R}^n$ into three regions: the set $H$ itself, along with two ``sides''.  If these sides are assigned values $+,-$ arbitrarily, then the family $\mathcal{H}$ (with, say, $|\mathcal{H}|=m$) partitions $\mathbb{R}^n$ into $3^m$ regions -- or ``faces'' -- with two points in the same region according to whether or not they satisfy the same relationships with respect to every $H\in \mathcal{H}$ (namely, for each $H\in \mathcal{H}$, either both lie in $H^+$, both in $H^-$, or both on $H$).  Such a family of regions carries a natural multiplication, with $F\cdot G$ denoting the region reached from any point $p$ of $F$ by travelling any sufficiently small distance in a straight line toward any point $q$ of $G$.  Note that if $F$ consists of a region that does not intersect any of the hyperplanes in $\mathcal{H}$, then $F\cdot G=F$.  

This algebra is known as a \emph{hyperplane face monoid}: the operation $\cdot$ is associative, and in fact is a left regular band operation; see Howie \cite[\S4.4]{how}.  Subsemigroups of hyperplane face monoids correspond to families of hyperplane faces closed under $\cdot$; we refer to these as \emph{hyperplane face semigroups}.  In Margolis, Saliola and Steinberg~\cite{MSS} it was shown that the class of semigroups isomorphic to hyperplane semigroups corresponds to the quasivariety generated by the three element monoid~${\bf L}$:
\[
\begin{tabular}{c|ccc}
$\cdot$&0&+&$-$\\
\hline
0&0&+&$-$\\
+&+&+&+\\
$-$&$-$&$-$&$-$
\end{tabular}
\]
(Note that ${\bf L}$ would usually be denoted by $L^1$ in semigroup theory, as it is the result of adjoining an identity element $1$ to the two-element left zero semigroup, but in our case the element $0$ is notationally playing the role of $1$.  However ${\bf L}$ is used in \cite{MSS}, and it is notationally convenient to continue this in the present article.)
The article~\cite{MSS} also shows that the following infinite set $\Sigma_{\bf L}$ of axioms completely axiomatises the quasivariety generated by ${\bf L}$ (and no finite axiomatisation exists).
\begin{align*}
x(yz)&=(xy)z\\
xx&=x\\
xyx&=xy\\
\left\{\begin{matrix} xy=x\And yx=y\\
\And  z_1x=x\And z_{2n+1}y=y
\end{matrix}\right\}\bigand &
\left\{
\begin{matrix}
\bigand_{1\leq i\leq 2n-1} z_i x=z_iy\\
 \bigand_{1\leq i\leq n} 
    \left(
        \begin{matrix}
        z_{2i-1}z_{2i}=z_{2i}\\ 
        z_{2i+1}z_{2i}=z_{2i}
        \end{matrix}
    \right)
\end{matrix}
\right\}
\rightarrow x=y
\end{align*}
We will refer to the $2n+3$ variable quasi-equation in $\Sigma_{\bf L}$ by $\lambda_n$.

Note that ${\bf L}$ is (isomorphic to) ${\bf 3}_{\{\sqcup\}}$, the $\sqcup$-reduct of ${\bf 3}$, and hence is representable as a $\sqcup$-algebra of functions.  Moreover, because all hyperplane monoids embed into a power of ${\bf L}$, it follows that all hyperplane monoids are representable as $\sqcup$-algebras of functions.  It is convenient to make this representation explicit.
For a family of hyperplanes $\mathcal{H}$ (with each $H$ having $H^+$ and $H^-$ fixed), and a face $F$, define the \emph{partial characteristic function} $\chi_F:\mathcal{H}\to\{+,-\}$ by
\[
\chi_F(H)=
\begin{cases}
\text{undefined}&\text{ if }F\subseteq H\\
+&\text{ if }F\subseteq H^+\\
-&\text{ if }F\subseteq H^-
\end{cases}
\]
It is essentially trivial to verify that the operation $\cdot$ on faces corresponds precisely to $\sqcup$ on the corresponding partial characteristic functions. 

\begin{cor}\label{cor:sqcupaxioms}
The class $\Rep(\sqcup)$ of algebras representable as $\sqcup$-algebras of functions is precisely the class of semigroups in the quasivariety of ${\bf L}$, and coincides at the finite level with the class of semigroups isomorphic to hyperplane face semigroups. The class is completely axiomatised by $\Sigma_{\bf L}$, and no complete finite axiomatisation is possible in first order logic.
\end{cor}

\begin{proof}
Note that ${\bf 3}_{\{\sqcup\}}$ is ${\bf L}$, so that the first statement is just the $\tau=\{\sqcup\}$ case of Theorem \ref{thm:3tau}.  The remaining statements are from the main results of \cite{MSS}.
\end{proof}

An alternative proof of Corollary \ref{cor:sqcupaxioms}, based on the work of \cite{MSS} is as follows, and may be useful in aiding the reader to understand the laws in $\Sigma_{\bf L}$ and some of the effort used to prove them in the next section.  First, the quasivariety generated by ${\bf L}$ consists of representable algebras, so it suffices to show that every representable algebra lies in this quasivariety.  Next verify that the laws $\Sigma_{\bf L}$ are sound for representable $\sqcup$-algebras.  Indeed, the equational laws are easily verified.  For the implications $\lambda_n$, observe that $x\sqcup y=y$ and $y\sqcup x=x$ together assert that $x$ and $y$ have the same domain.  The laws $z_i\sqcup x=z_i\sqcup y$ ensure that $x$ and $y$ agree outside of the places where each $z_i$ are defined.  In particular, $x$ and $y$ agree outside of the intersection of the domains of the $z_i$; let us denote this intersection by $D$.  Now, all the $z_i$ are in agreement on $D$ because they are alternately extensions and restrictions of each other: $z_1\sqcup z_2=z_2$ asserts that $z_1$ is a restriction of $z_2$, $z_3\sqcup z_2=z_2$ asserts that $z_2$ extends $z_3$, and so on.  But also, $x$ extends $z_1$ by $z_1\sqcup x=x$ and $y$ extends $z_{2n+1}$ by $z_{2n+1}\sqcup y=y$.  So $x$ and $y$ also agree with all the $z_i$ on $D$, and hence they agree everywhere that they are both defined.  Then $x=y$ as they have the same domain.

The variety generated by ${\bf L}$ is the variety of left regular bands, which is well known to be axiomatised by the first three laws of $\Sigma_{\bf L}$.  So we have the following corollary.

\begin{cor}\label{cor:weakcompsqcup}
The first three laws of $\Sigma_{\bf L}$ are weakly complete\textup: they are sound and complete for the equational theory of $\sqcup$-algebras of  functions.
\end{cor}

\section{Override and update}

We now turn to the main signature of interest -- override and update, $\{\sqcup,\underline{\phantom{a}}[\underline{\phantom{a}}]\}$ -- and solve the fundamental problem stated in \cite{minusover} and \cite{CLS}.   

To begin with we observe that axioms  \eqref{eq:absorb}--\eqref{eq:rightdist} (along with associativity and idempotence of $\sqcup$) that are presented in \cite{minusover} are not complete; similarly with those given in \cite[pp.~262]{CLS}.  Indeed, the following valid law is not a consequence:
\begin{align}
x[y[x[z]]]=x[y[x][z]].\label{eq:special}
\end{align}
Validity is easily proved; see Figure \ref{fig:3}.  \begin{figure}
\begin{tikzpicture}

\fill [pattern=north west lines]  (0,0)  circle  (1);
\scope
\clip 
      (0,0)  circle  (1);
   \clip   (0.625,1)  circle  (1);
\fill [white]  (1.25,0)  circle  (1);
\endscope
\scope
\clip 
      (0,0)  circle  (1);
   \clip   (0.625,1)  circle  (1);
\fill [pattern=north east lines]  (1.25,0)  circle  (1);
\endscope
\draw [ultra thick] (0,0)  circle  (1);
\draw [ultra thick] (1.25,0)  circle  (1);
\draw [ultra thick] (0.625,1)  circle  (1);
\node at (0.1,0.6) {$x$};
\node at (-1,-.9) {$x$};
\fill [white] (-.4,-.4) circle (.12);
\node at (-0.4,-0.4) {$x$};
\fill [white] (0.1,0.6) circle (.12);\node at (0.1,0.6) {$x$};
\fill [white] (0.62,-0.4) circle (.12);\node at (0.62,-.4) {$x$};
\fill [white] (0.62,.3) circle (.12);\node at (0.62,.3) {$z$};
\node at (-.5,1.5) {$y$};
\node at (2.2,-.9) {$z$};
\end{tikzpicture}
\caption{Both sides of the law  $x[y[x[z]]]=x[y[x][z]]$ yield this patterned Venn diagram.  The domain of definition is that of $x$, but the pattern is that of $z$ on the intersection of the three domains.}\label{fig:3}
\end{figure}
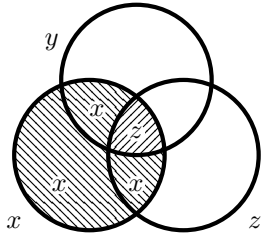
{}{\em Mace4} quickly finds a 16-element counterexample demonstrating the underivability of \eqref{eq:special} from \eqref{eq:absorb}--\eqref{eq:rightdist}. 
A further law of similar structure to that of \eqref{eq:special} is
\begin{equation}
w[x[y[w[z]]]]=w[x[y[w][z]]].\label{eq:special2}
\end{equation}
Equation \eqref{eq:special2} is also routinely shown to be sound, however we use it only to simplify our proof and note in Remark~\ref{rem:redundancies} that it is in fact redundant.

\begin{thm}\label{thm:overrideupdate}
The laws \eqref{eq:absorb}\up{--}\eqref{eq:rightdist}, \eqref{eq:special}, \eqref{eq:special2} along with idempotence and associativity of $\sqcup$ constitute a complete equational axiomatisation for the class $\Rep(\{\sqcup,\underline{\phantom{a}}[\underline{\phantom{a}}]\})$ of algebras in the signature $\{\sqcup,\underline{\phantom{a}}[\underline{\phantom{a}}]\}$ representable as systems of  functions.
\end{thm}
\begin{remark}\label{rem:redundancies}
There are redundancies in the axioms given in Theorem \ref{thm:overrideupdate} and a complete axiomatisation for $\{\sqcup,\underline{\phantom{a}}[\underline{\phantom{a}}]\}$ can be given by the laws \eqref{eq:absorb}, \eqref{eq:absorb2}, \eqref{eq:leftdist}, \eqref{eq:special} along with idempotence and associativity of $\sqcup$.  These redundancies can be easily established using {\em Prover9}.  We have stated Theorem \ref{thm:overrideupdate} with the larger redundant set, as these are the laws used in the proof presented here.  The {\em Prover9} proofs for the redundancy claim are not excessively complex, and human readable proofs could be obtained, but would add unnecessary routine technical tedium to the article.  The reader wishing to avoid further technical deductions can omit Lemmas~\ref{lem:firstlot} and~\ref{lem:secondlot}, provided that the laws proved there are added to the list of axioms in Theorem \ref{thm:overrideupdate}.
\end{remark}
The proof of Theorem~\ref{thm:overrideupdate} covers the rest of the section.  We prove completeness by first deriving a complete axiomatisation from the results of Section~\ref{sec:override}.  This axiomatisation is infinite and involves quasiequations as well as equations.  We then show that these infinitely many laws are consequences of the axioms in Theorem~\ref{thm:overrideupdate}.  

We begin by recalling that by Lemma \ref{lem:updatedef1}, the update operation is abstractly definable from override, using the formula $\phi_1(x,y,z)$.  Thus by Theorem~\ref{thm:abstractdef}, any complete axiomatisation of $\sqcup$ (such as $\Sigma_{\bf L}$) can be expanded to a complete axiomatisation of $\sqcup, \underline{\phantom{a}}[\underline{\phantom{a}}]$ by the addition of the law $\phi_1(x[y],x,y)$, which as a conjunction of equations, is equivalent to the following axioms (replacing the solution~$x$ in \eqref{eq:samedomain},~\eqref{eq:agree},~\eqref{eq:agree2} by the value $x_1[x_2]$, and renaming $x_1$ as $x$ and $x_2$ as $y$):
\begin{align}
x[y]\sqcup x&=x[y],\label{eq:domaineq1}\\
x\sqcup x[y]&=x,\label{eq:domaineq2}\\ 
x[y]\sqcup y&=y\sqcup x[y],\label{eq:agreement1}\\
y\sqcup x[y]&=y\sqcup x.\label{eq:agreement2}
\end{align}
We state this as a proposition.
\begin{pro}\label{pro:overrideupdate}
A sound and complete axiomatisation for $\Rep(\{\sqcup,\underline{\phantom{a}}[\underline{\phantom{a}}]\})$ is obtained by adjoining \eqref{eq:domaineq1}\up{--}\eqref{eq:agreement2} to $\Sigma_{\bf L}$.  
\end{pro}

Our goal now is to show that the laws in Proposition \ref{pro:overrideupdate} follow from the laws in Theorem \ref{thm:overrideupdate}.  A consequence of this is that there is a finite subset of $\Sigma_L$ which in the presence of \eqref{eq:domaineq1}--\eqref{eq:agreement2} are sufficient to yield all of $\Sigma_L$, though we are unsure at this point which finite subset is sufficient.  Surprisingly, {\em Prover9} finds that $\lambda_1$ and $\lambda_2$ are consequences of  \eqref{eq:absorb}--\eqref{eq:rightdist}, yet $\lambda_3$ is not.   We were able to use {\em Prover9} to demonstrate proofs that  $\lambda_3$--$\lambda_7$ were consequences of the full axiom system in Theorem \ref{thm:overrideupdate} (which involves the extra law \eqref{eq:special}), and even decipher very long human-readable proofs from these in the case of $\lambda_3$ and $\lambda_4$.  These intricate proofs did not provide many hints for a general proof for arbitrary $\lambda_n$.  The main influence on the final approach we present (for all $\lambda_n$) was confidence that the axioms in Theorem~\ref{thm:overrideupdate} were likely correct, and that deep nesting of $\underline{\phantom{a}}[\underline{\phantom{a}}]$ may play a role.

\begin{lem}\label{lem:firstlot}
The laws in Theorem \ref{thm:overrideupdate} imply \eqref{eq:domaineq1}\up{--}\eqref{eq:agreement2} as well as the equational laws in $\Sigma_L$ \up(associativity, idempotence and $x\sqcup y \sqcup x=x\sqcup y$\up).
\end{lem}
\begin{proof}
Idempotence and associativity of $\sqcup$ are immediate as they are included in Theorem \ref{thm:overrideupdate}. For $x\sqcup y \sqcup x=x\sqcup y$, observe that $(x\sqcup y)\sqcup x\refe{eq:absorb2}y[x]\sqcup x\sqcup x=y[x]\sqcup x\refe{eq:absorb2}x\sqcup y$.  Next, law \eqref{eq:domaineq2} follows by way of $x=x\sqcup x\refe{eq:absorb}x[x\sqcup y]\sqcup x\refe{eq:leftdist}x[y][x]\sqcup x\refe{eq:absorb2}x\sqcup x[y]$.  
For \eqref{eq:agreement2} use $x\sqcup y \refe{eq:absorb} (x\sqcup y)[(x\sqcup y)\sqcup y]\refe{eq:rightdist}x[(x\sqcup y)\sqcup y]\sqcup y[(x\sqcup y)\sqcup y]\refe{eq:absorb}x\sqcup y[(x\sqcup y)\sqcup y]\refe{eq:leftdist}x\sqcup y[y\sqcup y][x]\refe{eq:absorb}x\sqcup y[x]$.
For \eqref{eq:agreement1} we have $y\sqcup x[y]\refe{eq:absorb2}x[y][y]\sqcup y\refe{eq:leftdist}x[y\sqcup y]\sqcup y=x[y]\sqcup y$.
For \eqref{eq:domaineq1} note that $x[y]\refe{eq:absorb}x[y][x[y]\sqcup x]\refe{eq:leftdist}x[x[y]\sqcup x\sqcup y]\refe{eq:leftdist}x[x\sqcup y][x[y]]\refe{eq:absorb}x[x[y]]$.  Using this we obtain $x[y]\sqcup x\refe{eq:absorb2}x[x[y]]\sqcup x[y]=x[y]\sqcup x[y]=x[y]$, as required.
\end{proof}


We also need the following consequences.
\begin{lem}\label{lem:secondlot}
The following laws are consequences of the axioms in Theorem \ref{thm:overrideupdate}.
\begin{align}
x[y][z]&=x[z][y[z]]\label{eq:updateright}\\
x[y]=x&\rightarrow y[x]=y\label{eq:switch}\\
x\sqcup z=z\And y\sqcup z=z&\rightarrow x[y]=x\label{eq:jump}\\
x\sqcup y=x\And y\sqcup x=y&\rightarrow x[u][y]=y\label{eq:domainequality}\\
x\sqcup y=x\And y\sqcup x=y\And v\sqcup x=v\sqcup y&\rightarrow u[x][v[y][w]]=u[y][v[y][w]]\label{eq:samedom}
\end{align}
\end{lem}
\begin{proof}
For Equation \eqref{eq:updateright} we have
\[
x[y][z]\refe{eq:leftdist}x[z\sqcup y]\refe{eq:absorb2}x[y[z]\sqcup z]\refe{eq:leftdist}x[z][y[z]].
\]
For \eqref{eq:switch}, assume $x[y]=x$ and use $y\refe{eq:absorb}y[y\sqcup x]\refe{eq:leftdist}y[x][y]\refe{eq:updateright}y[y][x[y]]=y[x]$ with the last step using $x[y]=x$ and $y[y]=y[y\sqcup y]\refe{eq:absorb}y$.
Now for \eqref{eq:domainequality}, assume
\begin{equation}
x\sqcup y=x\And y\sqcup x=y\tag{$*$}\label{eq:assumption}
\end{equation}
and obtain $y\refe{eq:assumption}y\sqcup x\refe{eq:absorb2} x[y] \sqcup y=x[y] \sqcup y[y]\refe{eq:rightdist}(x\sqcup y)[y]\refe{eq:assumption}x[y]\refe{eq:absorb}x[x\sqcup u][y]\refe{eq:leftdist}x[y\sqcup x\sqcup u]\refe{eq:assumption}x[y\sqcup u]\refe{eq:leftdist}x[u][y]$.

For \eqref{eq:jump} assume that $x,y,z$ satisfy the premise of \eqref{eq:jump}.  Then $x[z]=x[x\sqcup z]\refe{eq:absorb}x$ so that $x[y]=x[z][y]\refe{eq:leftdist}x[y\sqcup z]=x[z]=x$.
  
Finally, we prove \eqref{eq:samedom}.  Note that the premise of the implication easily leads to $v[z] \sqcup x = v[z] \sqcup y$ for any $z$.  Applying \eqref{eq:leftdist} twice gives $u[x][v[y][w]]=u[v[w\sqcup y]\sqcup x]$, and as $v[z] \sqcup x = v[z] \sqcup y$ for any $z$, we have $u[v[w\sqcup y]\sqcup x]=u[v[w\sqcup y]\sqcup y]\refe{eq:leftdist}u[y][v[y][w]]$ as required.
\end{proof}

We need further preliminary lemmas.  
\begin{lem}
The following two implication schema are consequences of the axioms in Theorem \ref{thm:overrideupdate}\textup:
\[
\bigand_{1\leq i\leq n-1} x_i[x_{i+1}]=x_i\rightarrow x_1[x_2[x_3[\dots [x_n[u]]\dots]]]=x_1[x_2[x_3[\dots [x_n[x_1[u]]]\dots]]]\tag{$\eta_{n}$}
\]
and 
\begin{multline*}
\bigand_{1\leq i\leq n-1} x_i[x_{i+1}]=x_i\\
\rightarrow x_1[x_2[x_3[\dots [x_n[u]]\dots]]]=x_1[(x_2[x_1])[(x_3[x_1])[\dots [(x_n[x_1])[u]]\dots]]]\tag{$\eta_n'$}
\end{multline*}
\end{lem}
\begin{proof}
We first prove the $\eta_n$ family, then the $\eta_n'$ family.  The proof is by induction on $n$.  The base case is straightforward using $x_2=x_2[x_1]$ (by \eqref{eq:switch}, from $x_1[x_2]=x_1$) as $x_1[x_2[u]]=x_1[x_2[x_1][u]]\refe{eq:special}x_1[x_2[x_1[u]]]$.  Now let $k\geq 2$ and assume $\eta_k$ is true.   Assume $x_1,\dots,x_{k+1}$ satisfy the premise of $\eta_{k+1}$.  

Then 
\begin{align*}
x_1[x_2[x_3[\dots [x_{k+1}[u]]\dots]]]=&x_1[x_2[x_3[\dots [x_{k}[x_{k+1}[x_2[u]]]]\dots]]]\text{ by $\eta_k$, starting at $x_2$}\\
=&x_1[x_2[x_3[\dots x_{k}[x_{k+1}[x_2[x_1][u]]]\dots]]]\text{ by $x_2=x_2[x_1]$}\\
=&x_1[x_2[x_3[\dots x_{k}[x_1[x_{k+1}[x_2[x_1][u]]]]\dots]]]\text{  by $\eta_k$ starting at $x_1$}\\
\refe{eq:special2}&x_1[x_2[x_3[\dots x_{k}[x_1[x_{k+1}[x_2[x_1[u]]]]]\dots]]]\\
=&x_1[x_2[x_3[\dots x_{k}[x_{k+1}[x_2[x_1[u]]]]\dots]]]\text{ by $\eta_k$ starting at $x_1$}\\
=&x_1[x_2[x_3[\dots x_{k}[x_{k+1}[x_1[u]]]\dots]]]\text{ by $\eta_k$ starting at $x_2$.}
\end{align*}
 
 Now we look at the $\eta_n'$ family.  Let $n\geq 2$ be an arbitrary integer and assume that $x_1,\dots,x_n$ satisfy the premise of $\eta_n'$.  We have
\begin{align*}
x_1[x_2[x_3[\dots x_{n-1}[x_n[u]]\dots]]]&=x_1[x_2[x_3[\dots x_{n-1}[x_n[x_1[u]]]\dots]]]\text{ by $\eta_n$}\\
&=x_1[x_2[x_3[\dots x_{n-1}[x_1[x_n[x_1[u]]]]\dots]]]\text{ by $\eta_{n-1}$}\\
&\dots=x_1[x_2[x_1[x_3[x_1[\dots x_{n-1}[x_1[x_n[x_1[u]]]]\dots]]]]]\text{ by $\eta_{2}$}\\
&\refe{eq:special}x_1[x_2[x_1[x_3[x_1[\dots x_{n-1}[x_1[x_n[x_1][u]]]\dots]]]]]\\
&\dots\refe{eq:special}x_1[x_2[x_1][x_3[x_1][\dots x_{n-1}[x_1][x_n[x_1][u]]]\dots]]]
\end{align*}
This completes the proof of $\eta_n'$, and hence of the lemma.
\end{proof}

\begin{proof}[Proof of Theorem \ref{thm:overrideupdate}]
It remains to show that the laws $\lambda_n$ are a consequence of the axioms in Theorem \ref{thm:overrideupdate}.  
We will assume that $x,y,z_1,\dots,z_{2n+1}$ satisfy the premise of $\lambda_n$.  The law $\lambda_n$ refers only to $\sqcup$, however there are number of atomic formul{\ae} within the premise  that are more conveniently expressed in terms of $\underline{\phantom{a}}[\underline{\phantom{a}}]$.  First observe that the properties $z_{2i-1}\sqcup z_{2i}=z_{2i}$ and $z_{2i+1}\sqcup z_{2i}=z_{2i}$ imply $z_{2i-1}[z_{2i+1}]=z_{2i-1}$ by Equation \eqref{eq:jump} (and then $z_{2i+1}[z_{2i-1}]=z_{2i+1}$ by symmetry, or by  \eqref{eq:switch}).  It turns out we need only this simpler property (essentially: the expression $z_{2i-1}\sqcup z_{2i+1}$ could be shown to replace  the role of $z_{2i}$, though we do not need this fact in the proof).  Thus we need only the odd index $z$-variables, which we now label as $v_1:=z_1$, $v_2:=z_{3}$,\dots, $v_{n+1}=z_{2n+1}$, and so can now assume that for each $i$ we have $v_i\sqcup x=v_i\sqcup y$ as well as $v_i[v_{i+1}]=v_i$ and that $v_1\sqcup x=x$ and $v_{n+1}\sqcup y=y$.  Our goal is to show that $x=y$.  Note that \eqref{eq:domainequality} implies that $x[u][y]=y$ (for any $u$) and by symmetry $y[u][x]=x$.  Also, the property $v_1\sqcup x=x$ implies that $x\refe{eq:absorb}x[x]=x[v_1\sqcup x]\refe{eq:leftdist}x[x][v_1]=x[v_1]$, which by symmetry gives $y=y[v_{n+1}]$ and then by \eqref{eq:switch} gives $v_{n+1}[y]=v_{n+1}$.

Now we have
\begin{align*}
x&=x[v_1]=x[v_1[v_2]]=\cdots\\
=&x[v_1[v_2[\dots [v_{n-1}[v_n]]\dots]]]\text{ (using $x=x[v_1]$ and $v_i=v_i[v_{i+1}]$)}\\
=&x[v_1[x][v_2[x][\dots [v_{n}[x][v_{n+1}]]\dots]]]\\
&\qquad \qquad \text{(by $\eta_{n+1}'$, with $u:=v_{n+1}$, and with $x_1=x$, $x_2:=v_1$,\dots, $x_{n+1}:=v_{n}$)}\\
=&x[v_1[x][v_2[x][\dots [v_{n}[x][v_{n+1}[y]]]\dots]]]\text{ (as  $v_{n+1}=v_{n+1}[y]$)}\\
\refe{eq:samedom}&\cdots\refe{eq:samedom}x[y][v_1[y][v_2[y][\dots [v_{n}[y][v_{n+1}[y]]]\dots]]]\\
\refe{eq:updateright}&x[y][v_1[y][v_2[y][\dots [v_{n-1}[y][v_{n}[v_{n+1}]][y]]\dots]]]\\
\refe{eq:updateright}&\cdots\refe{eq:updateright}(x[v_1[v_2[\dots [v_{n-1}[v_{n+1}]]\dots]]])[y]\refe{eq:domainequality}y,
\end{align*}
as required.  
\end{proof}

\section{Choice functions}\label{sec:update}  
We now return to the unproved claims in Section \ref{sec:basics}.  We begin by proving Theorem \ref{thm:venn}, which gives a canonical representation of the free algebras for the class $\operatorname{Rep}(\tau)$.  
\begin{proof}[Proof of Theorem \ref{thm:venn}]
By Theorem \ref{thm:3tau} we may assume that $F_\tau(S)$ is a subalgebra of $\mathcal{P}(A,B)$ for some sets $A,B$.  Recall that $S$ denotes the free generators, which then are functions from $A$ to $B$.  Let $\bar{S}=\{\bar{s}\mid s\in S\}$ be a disjoint copy of $S$ and let each $\bar{s}$ denote a function from $A$ to $B\times S$ defined by $a\mapsto (s(a),s)$ for each $a$ in the domain of $s$.  Let $F'_\tau(S)$ denote the subalgebra of $\mathcal{P}(A,B\times S)$ generated by $\bar{S}$.  Using the first projection from $B\times S$, Lemma \ref{lem:functionquotient}(1) shows that $F'_\tau(S)$ maps homomorphically onto $F_\tau(S)$.  This homomorphism maps $\bar{s}\mapsto s$ and so is bijective from the generators $\bar{S}$ to the free generators $S$.  By the universal mapping property for $F_\tau(S)$, it follows that $F'_\tau(S)$ is isomorphic to $F_\tau(S)$, with free generators $\bar{S}$.

Now consider the quotient of $B\times S$ corresponding to the second projection.  This induces a proper quotient of $\mathcal{P}(A,B\times S)$, however it separates the elements of $F'_S$.  To see this, first note that the natural map described in Lemma \ref{lem:functionquotient}(1) does not change the domain of any  function, only the range.  Second, note that two distinct elements $f\neq g$ of $F'_S$ with the same domain must differ at some element $a\in A$.  By a trivial induction argument on applications of operations from $\{\sqcup,\mathbin{@},-,\underline{\phantom{a}}[\underline{\phantom{a}}]\}$ on elements from $\bar{S}$, it follows that there are $s_f,s_g\in S$ with $f(a)=\bar{s}_f(a)$ and $g(a)=\bar{s}_g(a)$.  But $\bar{s}_f(a)=(s_f(a),s_f)$ and $\bar{s}_g(a)=(s_g(a),s_g)$.  Because $f(a)\neq g(a)$ we have $s_f\neq s_g$ and then the second projection from $B\times S$ to $S$ separates $f(a)=(s_f(a),s_f)\mapsto s_f$ from $g(a)=(s_g(a),s_g)\mapsto s_g$, as required.  It now follows that $F_\tau(S)$ is isomorphic to the subalgebra of $\mathcal{P}(A,S)$ with each free generator $s\in S$ corresponding to a  function from $A$ that maps its domain constantly to $s$.  To avoid cumbersome notation, we now identify  $F_\tau(S)$ with this subalgebra of $\mathcal{P}(A,S)$.

Define an equivalence relation $\equiv$ on $A$ by stating $x\equiv y$ if for all $s\in S$ either both $x,y\in \dom(s)$ or $x,y\notin \dom(s)$.  It is clear that on each block of this equivalence, each element of $F_\tau(S)$ acts identically: it is either completely undefined, or maps the entire block to some $s\in S$.  If we select a transversal $A'$ of $\equiv$ (that is, one element is selected from each block of $\equiv$) and then replace $A$ by the subset $A'$, then Lemma~\ref{lem:functionquotient}(2) shows that we have taken a quotient of $\mathcal{P}(A,S)$.  As each element of $F_\tau(S)$ is either entirely undefined or entirely constant on each block of $\equiv$, it follows that this quotient separates the elements of $F_\tau(S)$; without loss of generality then, we will assume that $A=A'$, or equivalently that $\equiv$ is the equality relation.  Now we may identify $A$ with a subset of the powerset $\wp(S)$: indeed, we may label each element $a$ of $A$ by the subset of elements of $S$ defined  at $a$.  The assumption that $\equiv$ is the equality relation implies that this labelling is an injective function from $A$ into $\wp(S)$.  Without loss of generality, we may identify $A$ with this subset of $\wp(S)$.  Note then that each free generator $s\in S$ is defined at precisely those elements of $A$ that contain  $s$.  This nearly completes the proof, except that $A$ might possibly only contain some of the subsets of $S$.  Applying Lemma \ref{lem:functionquotient}(2) again (with $X=\wp(S)$ and $X'=A$), we find that $F_\tau(S)$ is a homomorphic image of the subalgebra of $\mathcal{P}(\wp(S),S)$ described in the theorem statement, with generators mapping bijectively to the free generators of $F_\tau(S)$.  The desired conclusion now follows by the universal mapping property for $F_\tau(S)$.
\end{proof}

As mentioned in Section \ref{sec:basics}, Berendsen et al.~\cite[\S6]{minusover} prove a completeness result for the operations of $\{-,\sqcup\}$ with respect to the patterned Venn diagram interpretation (over finitely many variables): any legitimate patterning of any set of regions can be achieved using $\{-,\sqcup\}$; see Proposition~14 of~\cite{minusover}.  Under the canonical representation for the $\Rep(\{-,\sqcup\})$-free algebras (given in Theorem \ref{thm:venn}), this means that $F_{\{-,\sqcup\}}(S)$ coincides with $\Choice(S)$ as a $\{-,\sqcup\}$-algebra of functions.  The update operation in contrast appears quite weak, as it is incapable of even changing the domain of a function.  We now show that in any other sense it is the most expressive of the operations discussed above, since by using update alone we may achieve any legitimate patterning of the domain of a function.  More formally: under the canonical representation given by Theorem \ref{thm:venn}, we show that when   $S$ is finite, $F_{\underline{\phantom{a}}[\underline{\phantom{a}}]}(S)$ equals the family of all choice functions with domains $\{A\subseteq S\mid s\in A\}$, for $s\in S$.    It will be convenient to denote a domain of this form by $D_s$, or $D_i$ in the case that the free generator is indexed as $s_i$.

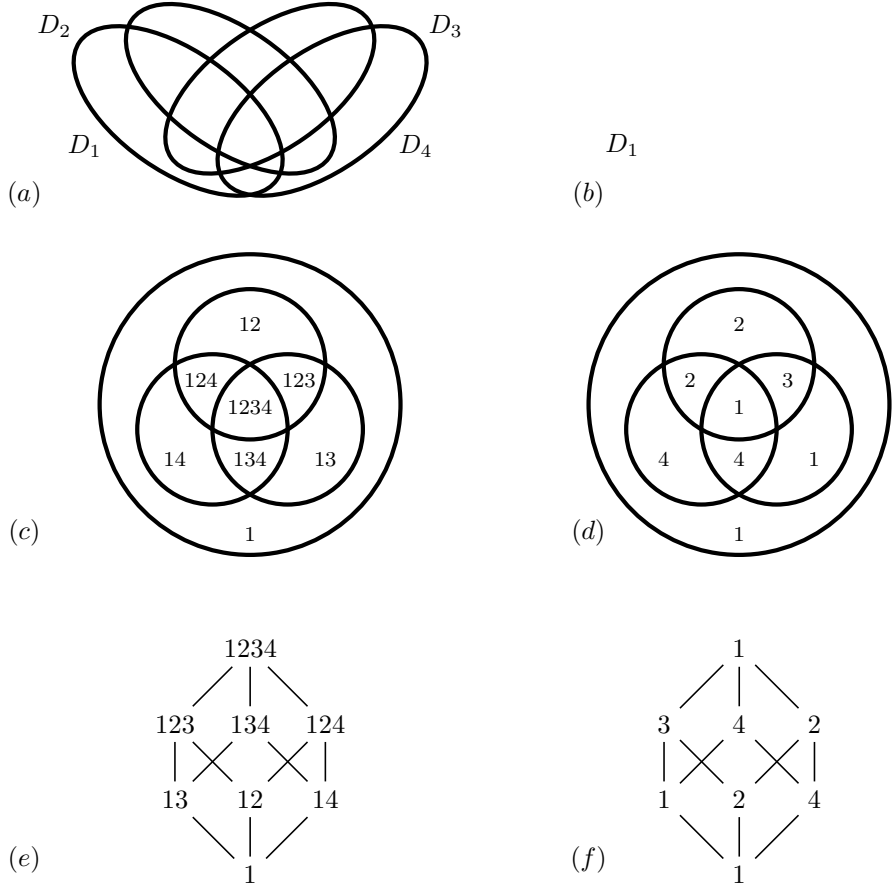
\begin{figure}
\begin{tikzpicture}
\node at (0,8) {
\begin{tikzpicture}[scale=0.8]
\draw [white] (-3.5,-.8) rectangle (3.5,3.2);
\draw [ultra thick, rotate = 35] (1.6,0.2) ellipse (2cm and 1cm);
\draw [ultra thick, rotate = 35] (1.1,1) ellipse (2cm and 1cm);
\draw [ultra thick, rotate = -35] (-1.6,0.2) ellipse (2cm and 1cm);
\draw [ultra thick, rotate = -35] (-1.1,1) ellipse (2cm and 1cm);
\node at (-2.75,0.5) {$D_1$};
\node at (-3.25,2.5) {$D_2$};
\node at (2.75,0.5) {$D_4$};
\node at (3.25,2.5) {$D_3$};
\end{tikzpicture}
};
\node at (6.5,8) {\begin{tikzpicture}[scale=0.8]
\draw  [white] (-3.5,-.8) rectangle (2,3.2);
\rotatebox{-35}{\scope
\clip
 (-1.6,0.2) ellipse (2.03cm and 1.03cm);
\draw [ultra thick, rotate = 70] (1.6,0.2) ellipse (2cm and 1cm);
\draw [ultra thick, rotate = 70] (1.1,1) ellipse (2cm and 1cm);
\draw [ultra thick, rotate = 0] (-1.6,0.2) ellipse (2cm and 1cm);
\draw [ultra thick, rotate = 0] (-1.1,1) ellipse (2cm and 1cm);
\endscope}
\node at (-2.75,0.5) {$D_1$};
\end{tikzpicture}};
\node at (0,4) {\begin{tikzpicture}
\draw [ultra thick] (0,0) circle [radius = 1];
\draw [ultra thick] (1,0) circle [radius = 1];
\draw [ultra thick] (60:1) circle [radius = 1];
\draw [ultra thick] (0.5,0.33) circle [radius = 2];
\node at (0.5,-1.4) {\footnotesize $1$};
\node at (0.5,1.4) {\footnotesize $12$};
\node at (1.5,-.4) {\footnotesize $13$};
\node at (.5,-.4) {\footnotesize $134$};
\node at (-.5,-.4) {\footnotesize $14$};
\node at (-.15,0.65) {\footnotesize $124$};
\node at (1.15,0.65) {\footnotesize $123$};
\node at (.5,0.3) {\footnotesize $1234$};
\end{tikzpicture}};
\node at (6.5,4) {\begin{tikzpicture}
\draw [ultra thick] (0,0) circle [radius = 1];
\draw [ultra thick] (1,0) circle [radius = 1];
\draw [ultra thick] (60:1) circle [radius = 1];
\draw [ultra thick] (0.5,0.33) circle [radius = 2];
\node at (0.5,-1.4) {\footnotesize $1$};
\node at (0.5,1.4) {\footnotesize $2$};
\node at (1.5,-.4) {\footnotesize $1$};
\node at (.5,-.4) {\footnotesize $4$};
\node at (-.5,-.4) {\footnotesize $4$};
\node at (-.15,0.65) {\footnotesize $2$};
\node at (1.15,0.65) {\footnotesize $3$};
\node at (.5,0.3) {\footnotesize $1$};
\end{tikzpicture}};
\node at (0,-.75) {\begin{tikzpicture}
\foreach \name/\x/\y/\pos  in {1234/2/3/above,123/1/2/left,134/2/2/right, 124/3/2/right, 12/2/1/below, 13/1/1/below, 14/3/1/below,1/2/0/below}
 		\node (\name) at (\x,\y) {\name};   	
 	 \foreach \src/\dest  in {1234/123,1234/124,1234/134,123/12,123/13,124/12,124/14,134/13,134/14,12/1,13/1,14/1} {\path (\src) edge[semithick]  (\dest); }
\end{tikzpicture}
};
\node at (6.5,-.75) {
\begin{tikzpicture}
\foreach \nodename/\name/\x/\y/\pos  in {1234/1/2/3/above,123/3/1/2/left,134/4/2/2/right, 124/2/3/2/right, 12/2/2/1/below, 13/1/1/1/below, 14/4/3/1/below,1/1/2/0/below}
 \node (\nodename) at (\x,\y) {\name};   	
 	 \foreach \src/\dest  in {1234/123,1234/124,1234/134,123/12,123/13,124/12,124/14,134/13,134/14,12/1,13/1,14/1} {\path (\src) edge[semithick]  (\dest); }
\end{tikzpicture}};

\node at (-3,6.8) {$(a)$};
\node at (4.5,6.8) {$(b)$};
\node at (-3,2.3) {$(c)$};
\node at (4.5,2.3) {$(d)$};
\node at (-3,-2.05) {$(e)$};
\node at (4.5,-2.05) {$(f)$};
\end{tikzpicture} 
\caption{(a): A Venn diagram of domains $D_1,D_2,D_3,D_4$ corresponding to each free generator $s_1,s_2,s_3,s_4$.  (b): An update term with first variable $x_1$ corresponds to a choice function whose domain consists of just one region: $D_1$.  (c): The domain $D_1$ redrawn, with each subregion labelled by the list of indices of free-generators that are defined there.  (d): An example choice function on the regions. (e): The domain $D_1$ as the ordered set of all subsets of $\{1,2,3,4\}$ containing $\{1\}$, corresponding to the diagram in (c).  (f): The choice function on this family of sets, corresponding to~(d).}\label{fig:venn4}
\end{figure}

We fix finitely many free generators $s_1,\dots,s_k$.  To avoid cumbersome notation in indexing we use the set $S=\{1,\dots,k\}$ in place of $\{s_1,\dots,s_k\}$ and represent each free generator $s_i$ as the (unique) constant choice function on domain $D_i=\{A\subseteq \{1,\dots,k\}\mid i\in A\}$.  
If $t(x_1,\dots,x_k)$ denotes a term in the language 
$\{\underline{\phantom{a}}[\underline{\phantom{a}}]\}$, with left-most variable $x_p$, say, then $t(s_1,\dots,s_k)$ determines a choice function on $D_p$: on each $A=\{i_1,\dots,i_m\}$ containing $p$, the function $t(s_1,\dots,s_k)$ must take one of the values $j\in \{i_1,\dots,i_m\}$ (in agreement with a free generator $s_{j}$).    Note that the Venn-diagrammatic representation of the set $\{i_1,\dots,i_m\}$ is the subregion $\left(\bigcap_{j\in A}x_j\right)\cap \overline{\left(\bigcup_{j\notin A}x_j\right)}$ (where overline denotes complementation); see Figure \ref{fig:venn4} for an example with $k=4$.  A term corresponding to the choice function in this figure is given after the proof of Theorem \ref{thm:VennComplete}, as an example of the inductive process used in the proof.

\begin{theorem}\label{thm:VennComplete}
Let $s_1,\dots,s_k$ be the free generators for $F_{\{\underline{\phantom{a}}[\underline{\phantom{a}}]\}}(S)$ in its canonical representation in $\Choice(\{1,\dots,k\})$\up; thus each $s_i$ is the constant choice function on domain $D_i$.
Then for  $p\in\{1,\dots,k\}$ and any choice function~$\gamma$ with domain $D_p$, there is a term $t(x_1,\dots,x_k)$ in the language $\{\underline{\phantom{a}}[\underline{\phantom{a}}]\}$ with $t(s_1,\dots,s_k)=\gamma$.
\end{theorem}
\begin{proof}
We work slightly more generally, by taking any choice function $\gamma$ on the domain $\wp^+(\{1,\dots,k\})$ of all \emph{nonempty} subsets of $\{1,\dots,k\}$.   We aim to construct an update term that agrees with the restriction of $\gamma$ to $D_p$.   We work inductively down $L:=\wp^+(\{1,\dots,k\})$.   For each $A\in L$, let $\gamma|_A$ denote the restriction of $\gamma$ to the upset ${\uparrow}A$ of $A$ in $L$.   For each $j$ in such an $A$ we will construct a term $t_{j\in A}$ with leftmost variable $x_j$ and such that the choice function $t_{j\in A}(s_1,\dots,s_k)$ agrees with $\gamma|_A$ on the elements of the upset of $A$ in $L$. 

The base case of the induction will be the top element $\top:=\{1,\dots,k\}$.  For each $j\in\{1,\dots,k\}$ we let $t_{j\in\top}$ be the term $x_j[x_{\gamma(\top)}]$.  
The function $t_{j\in\top}(s_1,\dots,s_k)=s_j[s_{\gamma(\top)}]$ has domain $D_j$, and coincides with $s_{\gamma(\top)}$ on $\bigcap_{1\leq i\leq k}D_i=\{\top\}$.  Hence $t_{j\in\top}(s_1,\dots,s_k)$ agrees with $\gamma$ on $\top$.

Now assume that for some $i\geq 0$ and every subset $A\subseteq \{1,\dots,k\}$ with $|A|=k-i$, and every $j\in A$ we have defined a term $t_{j\in A}$ that agrees with $\gamma|_{A}$ on ${\uparrow}A$ and has leftmost variable $x_j$.
If $i=k-1$ we are done, as then the $(k-i)$-element subsets are the singleton subsets, and $t_{p\in \{p\}}(s_1,\dots,s_k)$ agrees with $\gamma$ on the upset of $\{p\}$ (that is, on $D_p$), which completes the proof.  So now assume $i<k-1$ and consider some set $A\subseteq \{1,\dots,k\}$ with $|A|=k-(i+1)\geq 1$.  For each $j\notin A$, the set $A_j:=A\cup\{j\}$ has size $k-i$, so the induction hypothesis gives a term $t_{j\in A_j}$ agreeing with $\gamma$ on
$\operatorname{\uparrow}(A\cup\{j\})$.  There are precisely $i+1$ such choices for $j\notin A$, and we will enumerate them as $j_1,\dots,j_{i+1}$.  The induction step is based over the observation  that 
\[
{\uparrow}A=\{A\}\cup {\uparrow}A_{j_1}\cup\dots\cup {\uparrow}A_{j_{i+1}}.
\]
For each $\ell\in A$, define
\[
t_{\ell\in A}=x_{\ell}[x_{\gamma(A)}][t_{j_1\in A_{j_1}}]\dots[t_{j_{i+1}\in A_{j_{i+1}}}].
\]
The function $t_{\ell\in A}(s_1,\dots,s_k)$ has domain $D_\ell$ and agrees with $s_{\gamma(A)}$ on the region $\bigcap_{j\in A}D_j\cap \overline{\bigcup_{j\notin A}D_j}=\{A\}$ (that is, $t_{\ell\in A}(s_1,\dots,s_k)$ agrees with $\gamma$ at the point $A$) and by the induction hypothesis determines the same choice function as $\gamma$ on each set in ${\uparrow}A_{j_1},\dots, {\uparrow}A_{j_{i+1}}$.  Thus $\gamma_{t_{\ell\in A}}(s_1,\dots,s_k)$ agrees with $\gamma_{A}$ on ${\uparrow}A$, as required.
\end{proof}
As example of the proof method, we work through the example choice function in Figure \ref{fig:venn4}$(d,f)$, which has domain $D_1$ corresponding to $x_1$.  We obtain:
\begin{itemize}
\item $t_{i\in\{1,2,3,4\}}=x_i[x_1]$;
\item $t_{2\in\{1,2,3\}}$ is $x_2[x_3][x_4[x_1]]$ and $t_{3\in\{1,2,3\}}$ is $x_3[x_3][x_4[x_1]]$;
\item $t_{2\in\{1,2,4\}}$ is $x_2[x_2][x_3[x_1]]$ and $t_{4\in\{1,2,4\}}$ is $x_4[x_2][x_3[x_1]]$;
\item $t_{3\in\{1,3,4\}}$ is $x_3[x_4][x_2[x_1]]$ and $t_{4\in\{1,3,4\}}$ is $x_4[x_4][x_2[x_1]]$;
\item $t_{2\in\{1,2\}}$ is $x_2[x_2][t_{3\in\{1,2,3\}}][t_{4\in\{1,2,4\}}]=x_2[x_2][x_3[x_3][x_4[x_1]]][x_4[x_2][x_3[x_1]]]$;
\item $t_{3\in\{1,3\}}$ is $x_3[x_1][t_{2\in\{1,2,3\}}][t_{4\in \{1,3,4\}}]=x_3[x_1][x_2[x_3][x_4[x_1]]][x_4[x_4][x_2[x_1]]]$;
\item $t_{4\in\{1,4\}}$ is $x_4[x_4][t_{2\in\{1,2,4\}}][t_{3\in\{1,3,4\}}]=
x_4[x_4][x_2[x_2][x_3[x_1]]][x_3[x_4][x_2[x_1]]]$;
\item $t_{1\in\{1\}}$ is $x_1[x_1][t_{2\in\{1,2\}}][t_{3\in\{1,3\}}][t_{4\in\{1,4\}}]$,
\end{itemize}
which is
\begin{multline*}
x_1[x_1][x_2[x_2][x_3[x_3][x_4[x_1]]][x_4[x_2][x_3[x_1]]]]\\
[x_3[x_1][x_2[x_3][x_4[x_1]]][x_4[x_4][x_2[x_1]]]]\\
[x_4[x_4][x_2[x_2][x_3[x_1]]][x_3[x_4][x_2[x_1]]]].
\end{multline*}
There are clearly redundancies in this construction argument, as all instances of $x_i[x_i]$ can be replaced by $x_i$, and the $t_{3\in\{1,3\}}$ term updates all further subsets of $\{1,2,3,4\}$ containing both $1$ and $3$, and subsequently  $t_{4\in\{1,4\}}$ updates all subsets containing both $1$ and $4$. Thus we need only the $x_2[x_2]=x_2$ part of $t_{2\in \{1,2\}}$ and the $x_3[x_1][t_{2\in\{1,2,3\}}]$ part of $t_{3\in \{1,3\}}$.  Then $t_{1\in\{1\}}$ simplifies to
\[
x_1[x_2]
[x_3[x_1][x_2[x_3][x_4[x_1]]]]
[x_4[x_2[x_3[x_1]]][x_3[x_4][x_2[x_1]]]].
\]

The number of choice functions on the upset of a singleton is easily seen to be $\prod_{1\leq i\leq n}i^{\binom{n}{i}}$, so that the following result is an immediate corollary.
\begin{cor}
The $n$-generated free algebra in the variety generated by ${\bf 3}_{\{\underline{\phantom{a}}[\underline{\phantom{a}}]\}}$ has 
\[
n\cdot \prod_{1\leq i\leq n}i^{\binom{n}{i}}
\]
elements.
\end{cor}

The power of expressibility of update terms easily extends to similar characterisations of the power of expressibility of other subsignatures of $\{\underline{\phantom{a}}[\underline{\phantom{a}}],\sqcup,@,-\}$ containing $\underline{\phantom{a}}[\underline{\phantom{a}}]$.  
For example, in the signature $\{@,\underline{\phantom{a}}[\underline{\phantom{a}}]\}$, we may use $@$ to form any desired intersection of domains, and then use update to pattern freely within these domains.  It is very easy to see that the intersection of the domains of free generators in $F_{\tau}(S)$ (as a subalgebra of $\Choice(S)$) is precisely the set of principal upsets in the ordered set $\wp^+(S)$ of nonempty subsets of $S$.   Thus, it follows from Theorem \ref{thm:VennComplete} that under the canonical representation into $\Choice(S)$, the free algebra $F_{\{@,\underline{\phantom{a}}[\underline{\phantom{a}}]\}}(S)$  is precisely the set of all choice functions in $\Choice(S)$ having a domain that is a principal upset in $\wp^+(S)$.   Similarly, $\sqcup$ can be used to take unions of domains, so that $F_{\{\sqcup, @,\underline{\phantom{a}}[\underline{\phantom{a}}]\}}(S)$ consists of all choice functions on upsets in $\wp^+(S)$, while $F_{\{\sqcup,\underline{\phantom{a}}[\underline{\phantom{a}}]\}}(S)$ consists of all choice functions on unions of domains of the form~$D_s$.

\section{Conclusions, open problems and subsequent work}
We have shown that the class $\Rep(\sqcup,\underline{\phantom{a}}[\underline{\phantom{a}}])$ of algebras that are isomorphic to systems of functions with override and update form a variety with a relatively simple finite axiomatisation.  The class $\Rep(\sqcup)$ of override alone forms a nonfinitely axiomatisable quasivariety, but is not a variety, though its equational theory has a finite axiomatisation.  

In future work we will develop a uniform approach to representability that will enable us to give complete axiomatisations for $\Rep(\tau)$ in all combinations of the other operations discussed in Section~\ref{prelim:otheroperations}, provided that restriction is expressible.  That approach does not extend to the critical case of override and update considered in the present work.  

In Section \ref{sec:update} we have given a concrete description of the free algebras for the class $\Rep(\underline{\phantom{a}}[\underline{\phantom{a}}])$ which shows that update has surprisingly strong expressive power, and perhaps some extra prominence amongst the  operations considered.  Yet many properties of update remain elusive.  The proof method described in Theorem \ref{thm:VennComplete} makes use of a kind of normal form for expressions in $\underline{\phantom{a}}[\underline{\phantom{a}}]$, and then also for subsignatures of $\{\underline{\phantom{a}}[\underline{\phantom{a}}],\sqcup,@,-\}$ containing it.  Yet there remains the lack of any clear process for simplifying terms, for understanding equivalence of terms, for determining the minimum size of equivalent terms, and nor do we have a complete axiomatisation for $\Rep(\underline{\phantom{a}}[\underline{\phantom{a}}])$.

A further line of investigation is to incorporate composition of functions into the signature wherever possible.  While all of these signatures can be expressed as term reducts of the signatures characterised by the authors in \cite{modrest} and in \cite{HJSz}, many of the combinations sit naturally with composition and remain unclassified.  In future work on signatures that include $\mathbin{@}$, we expect it to be straightforward to add composition to the signatures considered.

\end{document}